\documentclass[12pt,fleqn]{article}
\usepackage{mathtools,amsfonts,amsmath, amsthm,amssymb}
\usepackage{accents}
\usepackage{chngcntr}

\usepackage[active]{srcltx}

\usepackage{latexsym}
\usepackage{esint}

\usepackage[margin=22mm]{geometry} 
 \numberwithin{equation}{section}

\usepackage{color}
 \definecolor{db}{rgb}{0.0,0.0,0.8} 
\definecolor{dg}{rgb}{0.0,0.55,0.14}
\definecolor{dr}{rgb}{0.5,0,0.07}

  \usepackage{hyperref}
 \hypersetup{frenchlinks=true,colorlinks=true,linkcolor=db,citecolor=dg,
 bookmarks=true}

 \usepackage{enumerate}
\usepackage{authblk}

\newtheorem{theorem}{Theorem}[section]
\newtheorem{proposition}{Proposition}[section]

\newtheorem{lemma}[proposition]{Lemma}
\newtheorem{corollary}[proposition]{Corollary}
\theoremstyle{definition}

\theoremstyle{definition}

\theoremstyle{definition}

\theoremstyle{definition}

\theoremstyle{definition}

\theoremstyle{definition}
\newtheorem{remark}[theorem]{Remark}
\theoremstyle{definition}

\newtheorem{open-problem}{Open Problem}
\newcounter{step}

\newcommand{\rlemma}[1]{Lemma~\ref{#1}}
\newcommand{\rth}[1]{Theorem~\ref{#1}}

\newcommand{\rprop}[1]{Proposition~\ref{#1}}

\newcommand{\rsec}[1]{Section~\ref{#1}}

\def\be{\begin{equation}}
\def\ee{\end{equation}}
\def\bes{\begin{equation*}}
\def\ees{\end{equation*}}
\def\bt{\begin{theorem}}
\def\et{\end{theorem}}
\def\bpr{\begin{proposition}}
\def\epr{\end{proposition}}
\def\bl{\begin{lemma}}
\def\el{\end{lemma}}
\def\bc{\begin{corollary}}
\def\ec{\end{corollary}}
\def\br{\begin{remark}}
\def\er{\end{remark}}
\def\ben{\begin{enumerate}}
\def\bena{\begin{enumerate}[a)]}
\def\een{\end{enumerate}}
\def\bit{\begin{itemize}}
\def\iit{\end{itemize}}

\def\dist{\operatorname{dist}}

\def\deg{\operatorname{deg}}

\DeclareMathAlphabet{\mathonebb}{U}{bbold}{m}{n}

\def\R{{\mathbb R}}

\def\Z{{\mathbb Z}}

\usepackage{csquotes}

\def\fo{\forall\, }

\def\im{\imath}
\def\ve{\varepsilon}

\def\so{{\mathbb S}^1}

\def\sn{{\mathbb S}^N}

\newcommand{\wtu}{{\widetilde u}}
\newcommand{\wtw}{{\widetilde w}}
\newcommand{\wpp}{W^{1/p,p}}

\def\dist{\operatorname{dist}}

\def\deg{\operatorname{deg}}

\def\Dist{\operatorname{Dist}}

\newcommand{\wop}{W^{1/p,p}}
\newcommand{\woo}{W^{1,1}}

\makeatletter
\def\moverlay{\mathpalette\mov@rlay}
\def\mov@rlay#1#2{\leavevmode\vtop{%
		\baselineskip\z@skip \lineskiplimit-\maxdimen
		\ialign{\hfil$\m@th#1##$\hfil\cr#2\crcr}}}
\newcommand{\charfusion}[3][\mathord]{
	#1{\ifx#1\mathop\vphantom{#2}\fi
		\mathpalette\mov@rlay{#2\cr#3}
	}
	\ifx#1\mathop\expandafter\displaylimits\fi}
\makeatother

\newcommand{\cupdot}{\charfusion[\mathbin]{\cup}{\cdot}}

\date{\today}
\title{On the distance between homotopy classes in
  $W^{1/p,p}(\so;\so)$}
\author{Itai Shafrir}
\affil{Department of Mathematics, Technion - I.I.T., 32 000 Haifa, Israel}

\newcommand\blfootnote[1]{%
  \begingroup
  \renewcommand\thefootnote{}\footnote{#1}%
  \addtocounter{footnote}{-1}%
  \endgroup
}

\begin{document}
\maketitle
\begin{abstract}
For every $p\in(1,\infty)$ there is a natural notion of topological
degree for maps in  $W^{1/p,p}(\so;\so)$ which allows us to write
that space as
a disjoint union of classes, $W^{1/p,p}(\so;\so)=\bigcup_{d\in\Z}\mathcal{E}_d$. For every pair
$d_1,d_2\in \Z$, we show that the distance $\Dist_{\wop}({\mathcal
  E}_{d_1}, {\mathcal E}_{d_2}):=\sup_{f\in{\mathcal E}_{d_1}}\
\inf_{g\in{\mathcal E}_{d_2}}\ d_{\wop}(f, g)$ equals the minimal
$\wop$-energy in $\mathcal{E}_{d_1-d_2}$. In the special case $p=2$ we 
deduce  from the latter formula an {\em explicit} value: $\Dist_{W^{1/2,2}}({\mathcal
  E}_{d_1}, {\mathcal E}_{d_2})=2\pi|d_2-d_1|^{1/2}$. 
\end{abstract}
\section{Introduction}
 For any $1<p<\infty$ consider the space $W^{1/p,p}(\so;\so)$
 consisting of the measurable functions $f:\so\to\R^2$ satisfying
 $f(x)\in\so$ a.e.~and
 \begin{equation}
\label{eq:18}
 |f|_{W^{1/p,p}}:=\left(\int_{\so}\int_{\so}\frac{|f(x)-f(y)|^p}{|x-y|^{2}}\, dxdy\right)^{1/p}<\infty.
 \end{equation}
Although the functions in $W^{1/p,p}(\so;\so)$ are not necessarily
continuous, a notion of topological degree does apply to maps in this
space, based on the density of $C^\infty(\so;\so)$ in
$W^{1/p,p}(\so;\so)$. This is a special case of the concept of  topological degree for maps in VMO, that was developed by  Brezis
and Nirenberg~\cite{vmo} (following a suggestion of L. Boutet de
Monvel and O. Gabber \cite[Appendix]{bgp}). It is natural to use this
degree to decompose the space into disjoint classes
$\{\mathcal{E}_d\}_{d\in\Z}$ and then to define the
\enquote{minimal energy} in each class, via the semi-norm in
\eqref{eq:18}, that is
\begin{equation}
   \sigma_{p}(d):=\inf_{f\in\mathcal{E}_{d}} |f|_{W^{1/p,p}}\,.
 \end{equation}
A lower bound for $\sigma_{p}(d)$ follows from the following result of
Bourgain, Brezis and Mironescu~\cite{bbm} who proved that there exists a positive constant $C_p$ such that 
 \begin{equation}
\label{eq:21}
   |\deg f|\leq C_p |f|^p_{W^{1/p,p}}\,,~\forall f\in W^{1/p,p}(\so;\so).
 \end{equation}
Therefore, 
\begin{equation}
  \label{eq:20}
  \sigma_{p}(d)\geq c_1(p)|d|^{1/p}\,,~\forall d\in\Z,
\end{equation}
 (with $c_1(p)=(1/C_p)^{1/p}$). In fact, a generalization of
 \eqref{eq:21} to the space  $W^{N/p,p}(\sn;\sn)$, $N\geq 2$, was also
 proved in \cite{bbm} (see  \cite{bbn,ngu} for refinements of this formula).

\par In the special case $p=2$ an explicit formula for $\sigma_2(d)$
is available, namely, 
 \begin{equation}
   \label{eq:mu2}
\sigma_2(d)=2\pi|d|^{1/2}\,.
 \end{equation}
 An easy way to establish \eqref{eq:mu2} is by using the expansion of
 $f\in W^{1/2,2}(\so;\so)$ to Fourier series,
 $f(e^{\im\theta})=\sum_{n=-\infty}^\infty a_ne^{\im n\theta}$.
 Indeed, combining the two well-known formulas (see e.g.~\cite{B1}):
 \begin{equation*}
   |f|^2_{W^{1/2,2}}=4\pi^2\sum_{n=-\infty}^\infty |n||a_n|^2~\text{ and
   }~\deg f=\sum_{n=-\infty}^\infty n|a_n|^2
 \end{equation*}
 yields the inequality $4\pi^2|\deg f|\leq |f|^2_{W^{1/2,2}}$, for every $f\in W^{1/2,2}(\so;\so)$,
 while equality occurs, e.g., for $f_d(z)=z^d$.
\par The distance function $\dist_{W^{1/p,p}}(f,g)=|f-g|_{W^{1/p,p}}$
induces two natural notions of distance between any pair of 
classes $\mathcal{E}_{d_1},\mathcal{E}_{d_2}$:
\begin{align}
\label{eq:23}
\dist_{\wop} ({\mathcal E}_{d_1}, {\mathcal E}_{d_2}):=\inf_{f\in{\mathcal E}_{d_1}}\ \inf_{g\in{\mathcal E}_{d_2}}\ d_{\wop}(f, g)\,,\\
\intertext{and}
\label{eq:22}
\Dist_{\wop}({\mathcal E}_{d_1}, {\mathcal E}_{d_2}):=\sup_{f\in{\mathcal E}_{d_1}}\ \inf_{g\in{\mathcal E}_{d_2}}\ d_{\wop}(f, g)\,.
\end{align}  
Both quantities in \eqref{eq:23}--\eqref{eq:22} were studied in
\cite{bms_coron}. Regarding $\dist_{\wop}$ the picture is completely
clear; it was shown in \cite{bms_coron} (by a similar argument to the
one used in \cite{vmo} in the case $p=2$) that $\dist_{\wop}({\mathcal E}_{d_1},
{\mathcal E}_{d_2})=0$ for all $d_1,d_2\in\Z$, for every
$p\in(1,\infty)$.  On the other hand, for $\Dist_\wop$ only partial
results were obtained. While the upper bound
\begin{equation}
  \label{eq:25}
  \Dist_{\wop}(\mathcal{E}_{d_1}, \mathcal{E}_{d_2})\leq
  c_2(p)|d_2-d_1|^{1/p},~\forall d_1,d_2\in\Z
\end{equation}
 was proved in \cite[Thm.~3, item 2]{bms_coron}, estimates for the
 lower bound were obtained only  under some restrictions on $p$ and/or
$d_1,d_2$. As an example, it was proved in \cite[Prop.~7.3]{bms_coron}
that  
\begin{equation}
  \label{eq:24}
  \Dist_{W^{1/2,2}}({\mathcal E}_{d_1}, {\mathcal
  E}_{d_2})=2\pi|d_2-d_1|^{1/2}
\end{equation}
 when $d_2>d_1\geq 0$. In the present
paper we give a precise formula for $\Dist_{\wop}({\mathcal E}_{d_1},
{\mathcal E}_{d_2})$, 
that in the special case $p=2$ yields the {\em explicit} formula
\eqref{eq:24} for {\em all} $d_1,d_2$. 
\begin{theorem}
\label{th:main}
  For every $p\in(1,\infty)$ and all $d_1,d_2\in\Z$ we have
  \begin{equation}
    \label{eq:Distp}
     \Dist_{W^{1/p,p}} ({\cal E}_{d_1}, {\cal
       E}_{d_2}) =\sigma_{p}(d_2-d_1)\,.
  \end{equation}
 In particular, there exist two positive constants $c_1(p)<c_2(p)$
 such that
 \begin{equation}
   \label{eq:27}
   c_1(p)|d_2-d_1|^{1/p}\leq \Dist_{W^{1/p,p}} ({\cal E}_{d_1}, {\cal
       E}_{d_2}) \leq c_2(p)|d_2-d_1|^{1/p},~\forall d_1,d_2\in\Z.
 \end{equation}
\end{theorem}
Formula \eqref{eq:27} provides a positive answer to Open Problem~2
from \cite{bms_coron} in the case of dimension $N=1$. It is an
immediate consequence of \eqref{eq:Distp}, \eqref{eq:20} and
\eqref{eq:25}. Note also that \eqref{eq:Distp} confirms the symmetry
property, $\Dist_{W^{1/p,p}} ({\cal E}_{d_1}, {\cal
       E}_{d_2})=\Dist_{W^{1/p,p}} ({\cal E}_{d_2}, {\cal
       E}_{d_1})$, which is not clear a priori from the definition
     \eqref{eq:22} (thus providing support for a positive answer to
     \cite[Open Problem~1]{bms_coron}).

\par In the case $p=2$ we obtain easily by combining \eqref{eq:Distp}
with \eqref{eq:mu2}:
\begin{corollary}
  \label{cor:p=2}
 We have
\begin{equation}
 \label{eq:29}
  \Dist_{W^{1/2,2}}({\mathcal E}_{d_1}, {\mathcal
  E}_{d_2})=2\pi|d_2-d_1|^{1/2},~\forall d_1,d_2\in\Z.
\end{equation}
\end{corollary}

\begin{remark}
  \label{rem:sigma}
  Using a similar argument to the one used in the proof of
  \rprop{prop:upper} below, it is easy to see that
  \begin{equation}
    \label{eq:11}
    \sigma_p^p(d)\leq |d|\sigma_p^p(1),~\forall d\in\Z. 
  \end{equation}
It follows that we may take $c_2(p)=\sigma_p(1)$ in
\eqref{eq:27}. While for $p=2$ equality holds in \eqref{eq:11} (by
\eqref{eq:mu2}), we do not know whether this is the case for other
values of $p$.
\end{remark}
The upper bound in \eqref{eq:Distp} is the easier assertion. It
follows from a slight modification of the argument used in the
proof of item {\it 2} of \cite[Theorem 3]{bms_coron}, that is, the
estimate \eqref{eq:25}.  
The proof of the lower bound in \eqref{eq:Distp} is much more involved; it
uses some arguments introduced in \cite{bms_w11} to prove a
lower bound for $\Dist_{\woo(\Omega;\so)}$ where $\Omega$ is either a
bounded domain in $\R^N$ or a smooth compact manifold, e.g.,
$\Omega=\so$ (for the special   case  $\woo(\so;\so)$, a
slightly different argument was used earlier in
\cite{bms_coron}). In particular, as in \cite{bms_coron,bms_w11} we make use of
\enquote{zig-zag}-type functions in order to construct functions in
$\mathcal{E}_{d_1}$ that are \enquote{relatively hard to approximate}
by functions in $\mathcal{E}_{d_2}$. This is the content of
\rprop{prop:Dist} below, whose proof requires some new tools
 due to the nonlocal character of the
$\wop$-energy.  In order to state it we need to introduce
some notation.
\par
 We start with  a notation for arcs in $\so$. For every $\alpha<\beta$ let
 \begin{equation}
   \label{eq:Arc}
 \mathcal{A}(\alpha,\beta)=\{e^{i\theta}\,;\,\theta\in(\alpha,\beta)\},\;\mathcal{A}(\alpha,\beta]=\{e^{i\theta}\,;\,\theta\in(\alpha,\beta]\}\text{
   and } \mathcal{A}[\alpha,\beta]=\{e^{i\theta}\,;\,\theta\in[\alpha,\beta]\}.
 \end{equation}
 For any $n\ge 1$ we divide $\so$ 
 to $2n$ arcs by setting 
\begin{equation}
\label{eq:1}
  I_{2j}=\mathcal{A}\bigl(2j\pi/n, (2j+1)\pi/n\bigr]\text{
    and } 
I_{2j+1}=\mathcal{A}\bigl((2j+1)\pi/n, (2j+2)\pi/n\bigr]\,,
\end{equation}
for $j=0,1,\ldots,n-1$.
Define
$\widetilde T_n=\widetilde T^{(\alpha)}_n\in\text{Lip}(\so;\so)$ with $\deg \widetilde T_n=1$ by $\widetilde
T_n(e^{\im \theta})=e^{\im
  \tau_n(\theta)}$, with $\tau_n$ defined on $[0,2\pi]$ by setting $\tau_n(0)=0$ and 
\begin{equation}
\label{eq:taun}
  \tau_n'(\theta)= \begin{cases} n^\alpha&
  \theta\in(2j\, \pi/n, (2j+1)\, \pi/n]\\
             -(n^\alpha-2)&
             \theta\in      ((2j+1)\, \pi/n, (2j+2)\, \pi/n]
  \end{cases},\ j=0,1,\ldots,n-1\,,
\end{equation}
 where $\alpha$ is any number satisfying 
 \begin{equation}
   \label{eq:alpha}
\begin{cases}
\alpha\in(1-1/p,1) &\text{ if }p\geq 2\\
\alpha\in(1/p,1) & \text{ if  }1<p<2
\end{cases}.
 \end{equation}
 We fix a value of $\alpha$ satisfying \eqref{eq:alpha}.
A useful property of $\widetilde T_n$ is
 \begin{equation}
   \label{eq:useful}
d_{\so}(x,\widetilde T_n(x))\leq \frac{\pi}{n^{1-\alpha}}\,,\quad x\in\so\,,
 \end{equation}
 where $d_{\so}$ denotes the geodesic distance in $\so$.
 The next proposition gives a partial analogue of
 \cite[Prop.~1.3]{bms_w11} to the $\wop$-setting.
\begin{proposition}
  \label{prop:Dist}
For any $d_1\ne0$ let $f(z)=z^{d_1}$ and
define for each
$n\ge 1$, $f_n(z)=\widetilde T_n\circ f\in\mathcal{E}_{d_1}$.
Then, for every $d_2\in\Z$ the
sequence $\{f_n\}$ satisfies 
\begin{equation}
\label{eq:like-sa7}
\lim_{n\to\infty} \inf_{g\in\mathcal{E}_{d_2}} d_{\wop}(f_n,g)=\sigma_p(d_2-d_1)\,.
\end{equation}
\end{proposition}
It is clear that \rprop{prop:Dist} implies the inequality
\enquote{$\ge$} in \eqref{eq:Distp} when $d_1\neq 0$ (as we shall see
in Section~\ref{sec:proofs} below, the case $d_1=0$ is trivial).
\par 
The paper is organized as follows. In \rsec{sec:prel} we prove
some technical results needed for the proof of our main results. 
\rsec{sec:key-lemma} is devoted to the proof of a key lemma, essential
to the proof of 
\rprop{prop:Dist}. Finally, the proofs of \rprop{prop:Dist}
and \rth{th:main} are given in \rsec{sec:proofs}.
\subsubsection*{Acknowledgments.} The author is grateful to the anonymous
referee for many helpful comments  and
especially for his suggestions for the proofs of \eqref{eq:IAB} and
\eqref{eq:ineq-beta}, that are considerably simpler and more
elementary than the original ones.  The author is indebted to 
Haim Brezis and Petru Mironescu for many interesting discussions on
the problem studied in this paper. The research 
was supported by  the Israel Science Foundation (Grant
No. 999/13). Part of this work was done while the author was visiting the
University Claude Bernard Lyon 1. He thanks the Mathematics Department
for its hospitality. 
\section{Preliminaries}
\label{sec:prel}
  We recall the following elementary result (see \cite[Lemma~5.2]{bms_w11}):
\begin{lemma}
  \label{lem:ineqs1}
 Let $z_1$ and $z_2$ be two points in $\so$ satisfying, for some $\varepsilon\in(0,\pi/2)$,
 \begin{equation}
   \label{eq:55}
 d_{\so}(z_1,z_2)\in(\varepsilon,\pi-\varepsilon).
 \end{equation}
 If the vectors $v_1,v_2\in\R^2$ satisfy 
 \begin{equation}
   \label{eq:54}
v_j\perp z_j,\,j=1,2,
 \end{equation}
then
\begin{equation}
  \label{eq:52}
|v_1-v_2|\geq (\sin\varepsilon) |v_j|,~j=1,2.
\end{equation}
\end{lemma}
 The intuition beyond the above result is quite simple. Informally
 speaking, if the points
 $z_1,z_2\in\so$ are neither close to each other nor close to being
 antipodal points, then it is impossible for a pair of    
 nonzero vectors, $v_1$ and $v_2$, in the tangent spaces of $\so$ at
 $z_1$ and $z_2$, respectively, to be \enquote{almost parallel}  to each other.
The next lemma  can
be viewed as a \enquote{discrete} version of \rlemma{lem:ineqs1}, where
tangent vectors are replaced by chords.
\begin{lemma}
  \label{lem:elem}
 For any $\varepsilon\in(0,\pi/2)$ and  every four points
 $z_1,z_2,w_1,w_2\in\so$ such that
 \begin{equation*}
\text{either }
z_1\overline{w}_1,z_2\overline{w}_2\in\mathcal{A}(\varepsilon,\pi-\varepsilon)\quad\text{or}\quad z_1\overline{w}_1,z_2\overline{w}_2\in\mathcal{A}(\pi+\varepsilon,2\pi-\varepsilon)\,,
 \end{equation*}
we have:
 \begin{equation}
   \label{eq:elem}
    |(z_1-w_1)-(z_2-w_2)|^2\geq (\sin^2\varepsilon)\max\left\{|z_1-z_2|^2,|w_1-w_2|^2\right\}\,.
 \end{equation}
\end{lemma}
\begin{proof}
Without loss of generality assume that $z_1\overline{w}_1,z_2\overline{w}_2\in\mathcal{A}(\varepsilon,\pi-\varepsilon)$ and   write $z_j=e^{\im \varphi_j}$ and $w_j=e^{\im \psi_j}$ with
  $\varphi_j-\psi_j\in(\varepsilon,\pi-\varepsilon)$, $j=1,2$. We may
  also assume that $z_1\ne z_2$ and $w_1\ne w_2$; otherwise the result
  is clear.
We have
\begin{align*}
  z_1-z_2&=e^{\im\varphi_1}-e^{\im\varphi_2}=2\im\sin\left(\frac{\varphi_1-\varphi_2}{2}\right)
e^{\im(\varphi_1+\varphi_2)/2}\,,\\
w_1-w_2&=e^{\im\psi_1}-e^{\im\psi_2}=2\im\sin\left(\frac{\psi_1-\psi_2}{2}\right)
e^{\im(\psi_1+\psi_2)/2}\,.
\end{align*}
Therefore,
\begin{equation}
\label{eq:prod}
  (z_1-z_2)\cdot\overline{w_1-w_2}=|z_1-z_2||w_1-w_2|
\tau \exp\im\big((\varphi_1-\psi_1)/2+(\varphi_2-\psi_2)/2\big)\,,
\end{equation}
 with $\tau\in\{-1,1\}$. 
Since by our assumption
$(\varphi_1-\psi_1)/2+(\varphi_2-\psi_2)/2\in(\varepsilon,\pi-\varepsilon)$,
 we get from \eqref{eq:prod} that an argument of
 $(z_1-z_2)\cdot\overline{w_1-w_2}$ lies in either the interval 
 $(\varepsilon,\pi-\varepsilon)$ (if $\tau=1$) or
 $(\pi+\varepsilon,2\pi-\varepsilon)$ (if $\tau=-1$). In any case, an
 argument lies in $(\varepsilon,2\pi-\varepsilon)$, whence
 \begin{equation*}
    |(z_1-w_1)-(z_2-w_2)|^2\geq |z_1-z_2|^2+|w_1-w_2|^2-2(\cos\varepsilon)|z_1-z_2||w_1-w_2|\,,
 \end{equation*}
 and \eqref{eq:elem} follows.
\end{proof}

\par
We will also   need the following result about Lipschitz self-maps of $\so$.
\begin{lemma}
  \label{lem:lip}
Let $k\in \text{Lip}[0,2\pi]$ with Lipschitz constant $L$ such that
$k(0)=k(2\pi)$. Define $K:\so\to\so$ by 
  $K(e^{\im \theta})=e^{\im k(\theta)}$, $\theta\in[0,2\pi]$.
Then,
 \begin{equation}
   \label{eq:47}
 \|K\|_{\text{Lip}}:=\sup_{\substack{x,y\in\so\\x\ne y}}
\frac{|K(x)-K(y)|}{|x-y|}\leq\max\{1,L\}\,.
 \end{equation}
\end{lemma}
\begin{proof}
  For any pair $\theta_1\ne \theta_2$ in $[0,2\pi)$ we have
  \begin{equation}
\label{eq:4}
    \frac{|K(e^{\im \theta_2})-K(e^{\im \theta_1})|}{|e^{\im
        \theta_2}-e^{\im \theta_1}|}=
\left|\frac{\sin\big((k(\theta_2)-k(\theta_1))/2\big)}{\sin\big((\theta_2-\theta_1)/2\big)}\right|
\leq \sup\left\{\frac{|\sin\theta|}{\sin t}\,;\,
  t\in(0,\pi/2],\,|\theta|\leq Lt\right\}.
  \end{equation}
 Fix any $t\in(0,\pi/2]$. We distinguish two cases: either $Lt\leq\pi/2$ or $Lt>\pi/2$. In the
 first case we have 
 \begin{equation}
\label{eq:48}
   \sup\left\{\frac{|\sin\theta|}{\sin t}\,;\,
  |\theta|\leq Lt\right\}=\frac{\sin(Lt)}{\sin t}\leq\max\{L,1\}.
 \end{equation}
 Indeed, if $L\le 1$ then clearly $\sin(Lt)\leq\sin(t)$. On the other hand, if $L>1$ then we use the fact that the function $g(t)=\sin(Lt)-L\sin t$ satisfies $g(0)=0$ and $g'(t)=L(\cos(Lt)-\cos t)\leq 0$ for $0\leq t\leq Lt\leq \pi/2$. 
 In the second case (where we must have $L>1$),
\begin{equation}
\label{eq:49}
   \sup\left\{\frac{|\sin\theta|}{\sin t}\,;\,
  |\theta|\leq Lt\right\}=\frac{1}{\sin t}<\frac{1}{\sin\big(\pi/(2L)\big)}<L,
 \end{equation}
 where the last inequality follows from the easily verified fact that
 the function $h(L):=L \sin\big(\pi/(2L)\big)$ satisfies $h(1)=1$ and
 $h'(L)>0$ on $[1,\infty)$. The conclusion \eqref{eq:47} clearly
 follows from \eqref{eq:48}--\eqref{eq:49}.
\end{proof}

\section{A key lemma}
\label{sec:key-lemma}
 It will be useful to introduce the following notation for $f\in
 W^{1/p,p}(\so;\so)$ and $A\subset\so\times\so$,
 \begin{equation*}
  E_p(f;A):=\iint_{A}\frac{|f(x)-f(y)|^p}{|x-y|^{2}}\, dxdy\,,
 \end{equation*}
  so in particular $E_p(f;\so\times\so)=|f|^p_{W^{1/p,p}}$.
\par The next lemma is the main  ingredient in the proof of \rprop{prop:Dist}.
\begin{lemma}
  \label{lem:main}
 Let $u,{\widetilde u},v\in \wpp(\so;\so)\cap C(\so;\so)$,
 $\varepsilon\in(0,\pi/20)$,
 and 
\begin{equation}
   \label{eq:83}
   \begin{aligned}
     C_\varepsilon^{+}&=\{x\in\so;\,  (v/\wtu)(x)\in\mathcal{A}[-\varepsilon,\varepsilon]\}\,,\\
 C_\varepsilon^{-}&=\{x\in\so;\,
 (v/\wtu)(x)\in\mathcal{A}[\pi-\varepsilon,\pi+\varepsilon]\}\,,\\
C_\varepsilon&=C_\varepsilon^{+}\cup C_\varepsilon^{-}\,,\\
D_\varepsilon&=\so\times\so\setminus\big( (C_\varepsilon^{+}\times
  C_\varepsilon^{+})\cup (C_\varepsilon^{-}\times
  C_\varepsilon^{-})\big)\,.
   \end{aligned}
 \end{equation}
Assume that 
\begin{equation}
\label{eq:14}
|u(x)-\widetilde u(x)|\leq \ve,\  \fo x\in\so,
\end{equation}
and let $\deg(u)=d_1$, $\deg(v)=d_2$.
Then, for some constant $c_1=c_1(p)>0$ we have, for $\varepsilon\leq\varepsilon_0(p)$,
\begin{equation}
\label{eq:34}
\begin{aligned}
E_p(v-\wtu;D_\varepsilon)\geq 
(1-c_1\ve^{1/2}) \sigma^p_p(d_2-d_1)- &c_1\varepsilon^{-p/2}
E_p(u;(\so\setminus C_\varepsilon)\times\so)\\
 -&c_1\varepsilon^{p/2} E_p(u;\so\times\so).
\end{aligned}
\end{equation}
\end{lemma}
\begin{proof}
  Note first that \eqref{eq:14} implies that
  $\deg(\wtu)=\deg(u)=d_1$. Hence,
  setting 
  $w:=v/u=v\, {\overline u}$ and $\wtw:=v/\wtu$, we have  $\deg(\wtw)=
  \deg(w)=d_2-d_1$. Consider the map
 \begin{equation}
\label{eq:Wn}
   W:=\overline{u}(v-\wtu)+1=w+(1-\wtu/u).
 \end{equation}
Since
\begin{equation*}
  W(x)-W(y)={\overline u}(x)\{(v(x)-\wtu(x))-(v(y)-\wtu(y))\}+
   ({\overline u}(x)-{\overline u}(y))(v(y)-\wtu(y))\,,
\end{equation*}
the triangle inequality yields,
\begin{equation}
\label{eq:12}
 |W(x)-W(y)|\leq |(v(x)-\wtu(x))-(v(y)-\wtu(y))|+|1-\wtw(y)||u(x)-u(y)|.
\end{equation}
Interchanging between $x$ and $y$ gives 
\begin{equation}
\label{eq:16}
 |W(x)-W(y)|\leq |(v(x)-\wtu(x))-(v(y)-\wtu(y))|+|1-\wtw(x)||u(x)-u(y)|.
\end{equation}
By \eqref{eq:12}--\eqref{eq:16} we have
\begin{multline}
\label{eq:188}
|W(x)-W(y)|\leq
  |(v(x)-\wtu(x))-(v(y)-\wtu(y))|+2|u(x)-u(y)|,\\
(x,y)\in \so\times\so\,,
\end{multline}
and
\begin{multline}
 \label{eq:17}
  |W(x)-W(y)|\leq
  |(v(x)-\wtu(x))-(v(y)-\wtu(y))|+\varepsilon|u(x)-u(y)|,\\
 (x,y)\in
               (C_\varepsilon^{+}\times\so)\cup
               (\so\times C_\varepsilon^{+}).
\end{multline}
Note that by \eqref{eq:83} $D_\varepsilon$ can be written as a disjoint union, 
\begin{equation}
\label{eq:disjoint}
  D_\varepsilon= ((\so\setminus C_\varepsilon)\times\so)\cupdot(C_\varepsilon\times(\so\setminus
   C_\varepsilon))\cupdot  (C_\varepsilon^{+}\times C_\varepsilon^{-})\cupdot
               (C_\varepsilon^{-}\times C_\varepsilon^{+}).
\end{equation}
Next we will use the following elementary inequality:
\begin{equation}
  \label{eq:elementary}
  (a+b)^p\leq (1+\eta)^pa^p+(1+1/\eta)^pb^p\,,\quad\forall a,b,\eta,p>0\,.
\end{equation}
For the proof of \eqref{eq:elementary} it suffices to notice that $a+b\leq(1+\eta)a$ when $\eta a\ge b$, while  $a+b<(1+1/\eta)b$ when $\eta a< b$.
By  \eqref{eq:disjoint} and \eqref{eq:elementary}, applied to \eqref{eq:188}--\eqref{eq:17}  with $\eta=\sqrt{\varepsilon}$, we obtain
\begin{equation}
\label{eq:13}
E_p(v-\wtu;D_\varepsilon)\geq  \frac{E_p(W;D_\varepsilon)}{(1+\sqrt{\varepsilon})^p}
-
2(2/\sqrt{\varepsilon})^p
E_p\big(u;\so\times(\so\setminus C_\varepsilon)\big)
-2\varepsilon^{p/2}E_p(u; C_\varepsilon^{+}\times C_\varepsilon^{-}).
\end{equation}
\par By \eqref{eq:14}, $|W-w|=|1-\wtu/u|=|u-\wtu|\leq\varepsilon$ in
 $\so$. Hence
\begin{equation}
  \label{eq:60}
 \big||W|-1\big|\leq |W-w|\leq \varepsilon~\text{ in }\so,
\end{equation}
 and also
 \begin{equation}
   \label{eq:63}
|\wtw-w|=|\wtu-u|\leq \varepsilon~\text{ in }\so.
 \end{equation}
Consider the map $\widetilde W:=W/|W|$,  which thanks to
\eqref{eq:60}  belongs to
$\wpp(\so;\so)$. Furthermore, again by \eqref{eq:60},
\begin{equation}
\label{eq:65}
  |\widetilde W-w|\leq |\widetilde W-W|+|W-w|\leq 2\varepsilon\ \text{in }\so,
\end{equation}
implying in particular that 
 \begin{equation}
\label{eq:62}
\deg(\widetilde W)=d_2-d_1.
\end{equation}
Combining \eqref{eq:65} with \eqref{eq:63} yields
 \begin{equation}
\label{eq:64}
   |\widetilde W-\wtw|\leq 3\varepsilon\ \text{and}\  d_{\so}(\widetilde W,\wtw)\leq 6\varepsilon\ \text{in }\so.
 \end{equation}
 From \eqref{eq:60} we get in particular that $|W|\geq 1-\varepsilon$,
 whence, 
using  the identity
\begin{equation*}
  |z_1-z_2|^2=(|z_1|-|z_2|)^2+|z_1|\cdot|z_2|\left|\frac{z_1}{|z_1|}-\frac{z_2}{|z_2|}\right|^2\,,\quad\forall z_1,z_2\in\mathbb{C}-\{0\}\,,
\end{equation*}
we get that
\begin{equation}
  \label{eq:144}
  |W(x)-W(y)| \geq (1-\varepsilon)|\widetilde W(x)-\widetilde
  W(y)|,\quad\forall x,y\in\so.
\end{equation}
Plugging \eqref{eq:144} in \eqref{eq:13} yields  
\begin{equation}
\label{eq:61}
\begin{aligned}
E_p(v-\wtu;D_\varepsilon)\geq  \Big(\frac{1-\varepsilon}{1+\sqrt{\varepsilon}}\Big)^pE_p(\widetilde W;D_\varepsilon)
&-
2(2/\sqrt{\varepsilon})^p
E_p(u;\so\times(\so\setminus C_\varepsilon))\\
&-2\varepsilon^{p/2}E_p(u; C_\varepsilon^{+}\times C_\varepsilon^{-}).
\end{aligned}
\end{equation}
By \eqref{eq:64} and \eqref{eq:83} we have 
\begin{equation}
  \label{eq:67}
\begin{aligned}
 C_\varepsilon^{+}&\subset \widetilde C_{\varepsilon}^{+}:=\{x\in\so;\,\widetilde
 W(x)\in\mathcal{A}[-7\varepsilon,7\varepsilon]\}\\
C_\varepsilon^{-}&\subset \widetilde C_{\varepsilon}^{-}:=\{x\in\so;\,\widetilde
 W(x)\in\mathcal{A}[\pi-7\varepsilon,\pi+7\varepsilon]\}
\end{aligned}
.
\end{equation}
Therefore,
\begin{equation}
  \label{eq:19}
  \widetilde D_\varepsilon:=\so\times\so\setminus\big( (\widetilde C_\varepsilon^{+}\times
  \widetilde C_\varepsilon^{+})\cup (\widetilde C_\varepsilon^{-}\times
  \widetilde C_\varepsilon^{-})\big)\subset D_\varepsilon.
\end{equation}
For each $\delta\in(0,\pi/2)$ we define (as in \cite{bms_w11})  the
map $K_\delta:\so\to\so$ by $K_\delta(e^{\im \theta})=e^{\im
  k_\delta(\theta)}$ where $k_\delta:[0,2\pi]\to[0,2\pi]$ is given by
\begin{equation}
\label{eq:Keps}
k_\delta(\theta):=\begin{cases}
0,&\text{if }\theta \in(0,\delta)\cup[2\pi-\delta,2\pi]\\
\pi(\theta-\delta)/(\pi-2\delta),&\text{if } \theta\in(\delta,\pi-\delta)\\
\pi, &\text{if } \theta\in[\pi-\delta,\pi+\delta]\\
\pi+\pi(\theta-\pi-\delta)/(\pi-2\delta),& \text{if }\theta\in[\pi+\delta,2\pi-\delta)\\
\end{cases}.
 \end{equation} 
Clearly $K_\delta\in\text{Lip}(\so;\so)$  and
$\deg(K_\delta)=1$. Since 
$\|k'_\delta\|_\infty=\pi/(\pi-2\delta)$  we have by \rlemma{lem:lip},
\begin{equation*}
  \left|K_\delta(e^{\im \theta_2})-K_\delta(e^{\im
      \theta_1})\right|\leq \left(\frac{\pi}{\pi-2\delta}\right)\left|e^{\im \theta_2}-e^{\im
      \theta_1})\right|\,,\quad \forall \theta_1,\theta_2\in[0,2\pi].
\end{equation*}
Therefore, $w_1:=K_{7\varepsilon}\circ\widetilde
W$ satisfies 
   $\deg(w_1)=\deg(\widetilde W)=d_2-d_1$ and 
   \begin{equation}
     \label{eq:15}
     |w_1(x)-w_1(y)|\leq
                 \left(\frac{\pi}{\pi-14\varepsilon}\right)|\widetilde W(x)-\widetilde W(y)|,~\forall x,y\in\so. 
   \end{equation}
By definition of $\sigma_p$, \eqref{eq:15} and the definition of $K_{7\varepsilon}$ (see
\eqref{eq:Keps}) it follows, using also \eqref{eq:19} and the fact that $w_1$ is constant on $\widetilde C_\varepsilon^{+}$ and $\widetilde C_\varepsilon^{-}$, that
\begin{multline}
  \label{eq:177}
  \sigma_p^p(d_2-d_1)\leq E_p(w_1;\so\times \so)=E_p(w_1;\widetilde D_\varepsilon)\\\leq
  \left(\frac{\pi}{\pi-14\varepsilon}\right)^p E_p(\widetilde W;
  \widetilde D_\varepsilon)\leq 
  \left(\frac{\pi}{\pi-14\varepsilon}\right)^p E_p(\widetilde W;
  D_\varepsilon).
\end{multline}
Plugging  \eqref{eq:177} in \eqref{eq:61} yields \eqref{eq:34}, for
large enough $c_1$.
\end{proof}

\section{Proof of \rth{th:main}}
\label{sec:proofs}
We begin with the upper bound for $\Dist_{\wop}$:
\begin{proposition}
  \label{prop:upper} 
For every $d_1,d_2\in\Z$ we have
\begin{equation}
  \label{eq:31}
  \Dist_{\wop}(\mathcal{E}_{d_1},\mathcal{E}_{d_2})\leq \sigma_p(d_2-d_1).
\end{equation}
\end{proposition}
\begin{proof}
Let $f\in\mathcal{E}_{d_1}$ and $\varepsilon>0$ be given. We need to
prove the existence of $g\in\mathcal{E}_{d_2}$ satisfying
\begin{equation}
  \label{eq:32}
  |f-g|_{\wop}^p\leq \sigma_p^p(d_1-d_2)+\varepsilon.
\end{equation}
By \cite[Lemma~2.2]{bms_coron} every map in $\wop(\so;\so)$ can be
approximated by a sequence $\{f_n\}\subset C^\infty(\so;\so)$ such
that each $f_n$ is constant near some point. Therefore, without loss
of generality we may assume that the given $f$ satisfies  $f\equiv 1$
in $\mathcal{A}(\pi-\delta,\pi+\delta)$ for some small $\delta>0$.
 By definition of $\sigma_p(d_2-d_1)$ there exists
 $h\in\mathcal{E}_{d_2-d_1}$ satisfying
 \begin{equation}
   \label{eq:33}
   |h|_{\wop}^p\leq \sigma_p^p(d_2-d_1)+\varepsilon.
 \end{equation}
 By the density result mentioned above, we may assume that
 $h\equiv 1$ in $\mathcal{A}(-\eta,\eta)$, for some small $\eta>0$. 
 Next we invoke the invariance of $|\cdot|_{W^{1/p,p}}$ with respect to
 M\" obius transformations $\mathcal{M}$ that send $\so$ to itself (see
 \cite{pm}) to get that
 \begin{equation}
   \label{eq:mobius}
  |h|_{W^{1/p,p}}=|h\circ \mathcal{M}|_{W^{1/p,p}}.
 \end{equation}
 For each $n\geq 1$ let $\mathcal{M}_n$ be the unique M\" obius transformation
 that sends the ordered triple (with respect to the
positive orientation on $\so$) 
 $(e^{\im(\pi+1/n)},1,e^{\im(\pi-1/n)})$ to the ordered triple
 $(e^{-\im\eta},1,e^{\im\eta})$. Hence $\mathcal{M}_n$ is a self
 map of $\so$ satisfying 
 $\mathcal{M}_n(\mathcal{A}(\pi+1/n,3\pi-1/n))=\mathcal{A}(-\eta,\eta)$.
 Set
 $h_n=h\circ \mathcal{M}_n$. Clearly $\deg h_n=\deg h=d_2-d_1$ and by
 \eqref{eq:mobius} and \eqref{eq:33}, for each $n$,
 \begin{equation}
\label{eq:mobn}
   |h_n|_{W^{1/p,p}}^p=
   |h|_{W^{1/p,p}}^p\leq \sigma_p^p(d_2-d_1)+\varepsilon ~\text{ and }~\{x\in\so\,;\,h_n(x)\neq 1\}\subset
    \mathcal{A}(\pi-1/n,\pi+1/n)\,.
 \end{equation}
  For every $n$ set $g_n=fh_n\in\mathcal{E}_{d_2}$. By construction it
  is clear that for  $n>1/\delta$ we have $g_n-f=f(h_n-1)=h_n-1$ on
  $\so$. Therefore, \eqref{eq:32} holds with $g=g_n$ for such $n$. 
\end{proof}
The main ingredient in the proof of the lower bound for $\Dist_{\wop}$
is \rprop{prop:Dist}.
\begin{proof}[Proof of \rprop{prop:Dist}]
 Clearly it suffices to consider $d_2\neq d_1$ with $d_1>0$. Let a small $\varepsilon>0$ be given. In view of the upper bound of
 \rprop{prop:upper}, it suffices to show that there exists
 $N(\varepsilon)$ such that (for every sufficiently small $\varepsilon$):
 \begin{equation}
   \label{eq:35}
   |f_n-g|_{\wop}^p\geq
  \sigma_p^p(d_2-d_1)-\varepsilon^{1/3},\quad\forall
   g\in\mathcal{E}_{d_2},\,\forall n\ge N(\varepsilon).
 \end{equation}
Fix any $g\in\mathcal{E}_{d_2}$. By density of smooth maps in
$\wop(\so;\so)$ we may assume that $g\in C^{\infty}(\so;\so)$. Clearly
it suffices to consider $n$ for which
\begin{equation}
  \label{eq:2}
  |f_n-g|_{\wop}^p\leq \sigma_p^p(d_2-d_1).
\end{equation}

  Consider the map
  \begin{equation}
 \label{eq:Hn}
    H_n:=\bar{f}(g-f_n)+1=h+(1-\bar{f}f_n)\,.
  \end{equation}
 Put
 $N_1(\varepsilon):=\left[(\pi/\varepsilon)^{1/(1-\alpha)}\right]+1$. By
 \eqref{eq:useful} we deduce that
 \begin{equation}
   \label{eq:37}
   |f_n-f|\leq\varepsilon \text{ on }\so,~\forall n\geq N_1(\varepsilon).
 \end{equation}
 For such $n$ we may apply
 \rlemma{lem:main}  with $u=f,\wtu=f_n$ and $v=g$ to get that
\begin{equation}
\label{eq:36}   
|g-f_n|_{\wop}^p\geq 
(1-c_1\ve^{1/2}) \sigma^p_p(d_2-d_1)- c_1\varepsilon^{-p/2}
E_p(f;(\so\setminus C_\varepsilon^{(n)})\times\so)
 -c_1 \gamma_{d_1}\varepsilon^{p/2},
 \end{equation}
where  for each $d\in\Z$ we denote 
\begin{equation}
  \label{eq:44}
  \gamma_{d}:=|z^{d}|^p_{\wop}\,,
\end{equation}
and where
\begin{equation*}
  C_\varepsilon^{(n)}=\{x\in\so\,;\,(\bar f_ng)(x)\in
  \mathcal{A}[-\varepsilon,\varepsilon]\cup \mathcal{A}[\pi-\varepsilon,\pi+\varepsilon]\}.
\end{equation*}

\par In order to conclude via \eqref{eq:36} we need to bound the term
$E_p(f;(\so\setminus C_\varepsilon^{(n)})\times\so)$. We claim that
there exists $C=C(p,d_1,d_2)$ such that for some $\beta>0$ there holds
\begin{equation}
  \label{eq:claim}
  E_p(f;(\so\setminus C_\varepsilon^{(n)})\times\so)\leq\frac{C}{\varepsilon}n^{-\beta}.
\end{equation}
We may write 
$\so\setminus C_\varepsilon^{(n)}=A_{\varepsilon,+}^{(n)}\cup
A_{\varepsilon,-}^{(n)}$ where
\begin{equation*}
  A_{\varepsilon,+}^{(n)}=\{x\in\so\,;\,(\bar f_ng)(x)\in
  \mathcal{A}(\varepsilon,\pi-\varepsilon)\}~\text{ and }A_{\varepsilon,-}^{(n)}=\{x\in\so\,;\,(\bar f_ng)(x)\in
  \mathcal{A}(\pi+\varepsilon,2\pi-\varepsilon)\}.
\end{equation*}
 Next  we write $\so$ as a disjoint union of the $2nd_1$ arcs given by 
 \begin{equation*}
 \tilde I_k=\mathcal{A}\Big(\frac{k\pi}{nd_1}, \frac{(k+1)\pi}{nd_1}\Big]\,,~k=0,1,\dots,2nd_1-1.
 \end{equation*}
 By the definition of $f_n$ we have (for large $n$) for all $x\ne y$ in $\tilde I_k$:
 \begin{equation}
 \label{eq:Ism}
 \frac{d_{\so}(f_n(x),f_n(y))}{d_\so(x,y)}=\begin{cases}
 n^\alpha d_1 & k\text{ is even}\\
 (n^\alpha-2)d_1 & k\text{ is odd}
 \end{cases}
 .
 \end{equation}
 We use these arcs to write
 $A_{\varepsilon,+}^{(n)}=\displaystyle\bigcup_{k=0}^{2nd_1-1} J_{k,+}$ where
 $J_{k,+}=A_{\varepsilon,+}^{(n)}\cap \tilde I_k$.
Using  the following  basic relation between the geodesic and
  Euclidean distances in $\so$,
\begin{equation}
  \label{eq:9}
  \Big(\frac{2}{\pi}\Big)d_{\so}(x,y)\leq |x-y|\leq d_{\so}(x,y),\quad\forall x,y\in\so\,,
\end{equation}
we deduce from \eqref{eq:Ism}
that
\begin{equation}
\label{eq:fnxy}
\frac{|f_n(x)-f_n(y)|^p}{|x-y|^{2}}\geq C_1n^{\alpha p} |x-y|^{p-2}\,,~\text{ for all } x\neq y\text{ in } J_{k,+},
\end{equation}
for some constant
$C_1=C_1(p,d_1)$. Applying \eqref{eq:elem} with
$z_1=f_n(x),z_2=f_n(y),w_1=g(x)$ and $w_2=g(y)$ to the L.H.S.~of \eqref{eq:fnxy}, and then  integrating over $J_{k,+}\times J_{k,+}$  yields
\begin{equation}
  \label{eq:6}
  \begin{aligned}
  \iint_{J_{k,+}\times J_{k,+}}\!\!\!\!\frac{|(f_n(x)-g(x))-(f_n(y)-g(y))|^p}{|x-y|^{2}}&\,dx\,dy\geq
  C_1(\sin^p\varepsilon) n^{\alpha p} \iint_{J_{k,+}\times J_{k,+}}\!\!\!|x-y|^{p-2}\,dx\,dy ,\\
&\quad k=0,1,\ldots,2nd_1-1\,.
\end{aligned}
\end{equation}
Next, we can also write $A_{\varepsilon,-}^{(n)}=\displaystyle\bigcup_{k=0}^{2nd_1-1} J_{k,-}$ where 
$\{ J_{k,-}\}_{k=0}^{2nd_1-1}$
 are defined analogously to $\{ J_{k,+}\}_{k=0}^{2nd_1-1}$.
The same computation that led to \eqref{eq:6} gives 
\begin{equation}
  \label{eq:8}
  \begin{aligned}
   \iint_{J_{k,-}\times J_{k,-}}\!\!\!\!\frac{|(f_n(x)-g(x))-(f_n(y)-g(y))|^p}{|x-y|^{2}}&\,dx\,dy\geq
  C_1(\sin^p\varepsilon) n^{\alpha p} \iint_{J_{k,-}\times J_{k,-}}\!\!\!|x-y|^{p-2}\,dx\,dy ,\\
  &\quad k=0,1,\ldots,2nd_1-1\,.
  \end{aligned}
\end{equation}
 Summing over all indices in \eqref{eq:6}--\eqref{eq:8} and taking
 into account \eqref{eq:2} yields
 \begin{equation}
   \label{eq:38}
   \sum_{k=0}^{2nd_1-1}\iint_{J_{k,-}\times J_{k,-}}\!
   |x-y|^{p-2}\,dx\,dy+\sum_{k=0}^{2nd_1-1}\iint_{J_{k,+}\times J_{k,+}}\!
   |x-y|^{p-2}\,dx\,dy\leq \frac{C_2}{(n^\alpha \sin\varepsilon)^p}\,.
 \end{equation}
 Next we treat separately   the cases $p\geq 2$ and $1<p<2$.\\[2mm]
\underline{Case I: $p\geq 2$}\\[2mm]
The key tool in treating this case is the following elementary inequality: 
 \begin{equation}\label{eq:IAB}
\iint_{A\times A} |x-y|^a\,dx\,dy\geq \kappa_a |A|^{a+2}\,,~\forall A\subset\so,\,\forall a\ge0,
\end{equation}
for some constant $\kappa_a>0$.
[Obviously we consider only measurable subsets of $\so$ and $|A|$  denotes the one dimensional  Hausdorff measure of  $A$].
To verify \eqref{eq:IAB} we first note that for any measurable  set $A\subset\R$ the set  
$$
B:=\{x\in A\,;\, |x|\geq |A|/4\}\,,
$$
satisfies $|B|\geq |A|/2$ (here $|C|$ stands for the Lebesgue measure of  $C\subset\R$). It follows that 
\begin{equation}\label{eq:AB}
\int_A |x|^a\,dx\geq \int_B |x|^a\,dx\geq |B| (|A|/4)^a\geq \tilde c_a |A|^{a+1},~\forall a\geq 0.
\end{equation}
Since \eqref{eq:AB} is clearly invariant w.r.t~translations, we deduce that also 
\begin{equation*}
\int_A |x-y|^a\,dx\geq   \tilde c_a |A|^{a+1}\,,~\forall A\subset\R,\,\forall y\in\R,\,\forall a\ge0\,,
\end{equation*}
and an additional integration yields
\begin{equation}
\label{eq:onR}
\iint_{A\times A} |x-y|^a\,dx\,dy\geq   \tilde c_a |A|^{a+2}\,,~\forall A\subset\R,\,\forall a\ge0\,.
\end{equation}
 Switching  from $\so$ to $\R$, using \eqref{eq:9}, enables us to deduce \eqref{eq:IAB} from \eqref{eq:onR}.
\par Applying \eqref{eq:IAB} to $A=J_{k,\pm}$ and $a=p-2$ gives
\begin{equation}
  \label{eq:39}
  \iint_{J_{k,\pm}\times J_{k,\pm}}
   |x-y|^{p-2}\,dx\,dy\geq \kappa_{p-2}|J_{k,\pm}|^p\,.
\end{equation}
 Plugging   \eqref{eq:39} in \eqref{eq:38} yields 
 \begin{equation}
   \label{eq:40}
   \sum_{k=0}^{2nd_1-1}  \left(|J_{k,+}|^p+|J_{k,-}|^p\right)\leq \frac{C_3}{(n^\alpha \sin\varepsilon)^p}\,.
 \end{equation}
By H\" older inequality and \eqref{eq:40} we obtain,
\begin{equation}
  \label{eq:41}
  \big|\so\setminus C_\varepsilon^{(n)}\big|=\sum_{k=0}^{2nd_1-1}
  \left(|J_{k,+}|+|J_{k,-}|\right)\leq (4nd_1)^{1-1/p}
  \frac{C_3^{1/p}}{n^\alpha\sin\varepsilon}\leq
  \frac{C_4}{\varepsilon} n^{1-1/p-\alpha}\,.
\end{equation}
Finally, by \eqref{eq:41} we get
\begin{equation}
  \label{eq:42}
  E_p(f;(\so\setminus C_\varepsilon^{(n)})\times\so)\leq
  2\pi\left|\so\setminus
  C_\varepsilon^{(n)}\right|\sup_{\substack{x,y\in\so\\x\ne y}}\frac{|f(x)-f(y)|^p}{|x-y|^2}\leq \frac{C_5}{\varepsilon} n^{1-1/p-\alpha},
\end{equation}
 which gives \eqref{eq:claim} with $\beta=\alpha-(1-1/p)>0$ (by \eqref{eq:alpha}).\\[2mm]
\underline{Case II: $1<p<2$}\\[2mm]
Treating this case requires another elementary inequality, namely,  
 \begin{equation}\label{eq:ineq-beta}
\iint_{A\times A} \frac{dx\,dy}{|x-y|^b}\geq \lambda_b\left( \iint_{A\times \so} \frac{dx\,dy}{|x-y|^b}\right)^2\,,~\forall A\subset\so,\,\forall b\in(0,1),
\end{equation}
for some $\lambda_b>0$.
To confirm \eqref{eq:ineq-beta} we first notice that $\int_{\so}\frac{dy}{|x-y|^b}:=\eta=\eta(b),\,\forall x\in\so$, and thus 
\begin{equation}\label{eq:333}
\iint_{A\times\so}\frac{dx\,dy}{|x-y|^b}=\eta|A|\,,\text{ for every measurable $A\subset\so$}.
\end{equation}
Finally, by \eqref{eq:333}
\begin{equation*}
\iint_{A\times A}\frac{dx\,dy}{|x-y|^b} \geq \frac{1}{2^b}|A|^2=\frac{1}{2^b\eta^2}\left(\iint_{A\times\so}\frac{dx\,dy}{|x-y|^b}\right)^2\,,
\end{equation*}
and \eqref{eq:ineq-beta} follows with $\lambda_b=\frac{1}{2^b\eta^2}$.
\par Next we turn to the proof of \eqref{eq:claim} in this case.  Clearly
\begin{equation}
\label{eq:46}
E_p(f;(\so\setminus C_\varepsilon^{(n)})\times\so)\leq C_6
\sum_{k=0}^{2nd_1-1}\left(\iint_{J_{k,-}\times \so}\!\!
|x-y|^{p-2}\,dx\,dy+\iint_{J_{k,+}\times\so}\!\!
|x-y|^{p-2}\,dx\,dy\right)\,.
\end{equation}
Applying the Cauchy-Schwarz inequality to \eqref{eq:46} and using
\eqref{eq:ineq-beta} (with $A=J_{k,\pm}$ and $b=2-p$) and \eqref{eq:38} yields
\begin{equation*}
  E_p(f;(\so\setminus C_\varepsilon^{(n)})\times\so)\leq
  C_6(4nd_1)^{1/2}\frac{C_2^{1/2}}{\lambda_{2-p}^{1/2}(n^\alpha\sin\varepsilon)^{p/2}}\leq
  \frac{C_7}{\varepsilon}n^{(1-\alpha p)/2}\,,
\end{equation*}
 and \eqref{eq:claim} follows in this case as well, with
 $\beta=(\alpha p-1)/2>0$ (see \eqref{eq:alpha}).
\par Choosing
$N(\varepsilon)\geq N_1(\varepsilon)$ (see
\eqref{eq:37}) such that, in addition, 
\begin{equation*}
  Cn^{-\beta}\leq \varepsilon^{1+p},~\forall n\geq N(\varepsilon),
\end{equation*}
we get from \eqref{eq:36} and \eqref{eq:claim} that for $n\geq
N
(\varepsilon)$ there holds,
\begin{equation*}
|g-f_n|_{\wop}^p\geq
(1-c_1\varepsilon^{1/2})\sigma^p_p(d_2-d_1)-c_1\ve^{p/2}(1+\gamma_{d_1})
\geq \sigma^p_p(d_2-d_1)-\varepsilon^{1/3}\,,
 \end{equation*}
 for $\varepsilon$ sufficiently small (using $p/2>1/2$), and \eqref{eq:35} follows.
\end{proof}
 We can now give the proof of our main result \rth{th:main}.
\begin{proof}[Proof of \rth{th:main}] In view of \eqref{eq:31} of
  \rprop{prop:upper}, it suffices to prove that 
  \begin{equation}
\label{eq:5}
\Dist_{\wop}(\mathcal{E}_{d_1},\mathcal{E}_{d_2})\geq
      \sigma_p(d_2-d_1),\quad\forall d_1,d_2\in\Z.
  \end{equation}
In case $d_1\neq 0$, \eqref{eq:5} follows from \rprop{prop:Dist}. In
the remaining (easy) case $d_1=0$, we can take the constant function
$f=1$ that satisfies $d_{\wop}^p(f,g)=|g|_{\wop}^p\geq
\sigma_p^p(d_2)$ for all $g\in\mathcal{E}_{d_2}$.
\end{proof}


\begin{thebibliography}{6}
\bibitem{bbm}
J.~Bourgain, H.~Brezis, and P.~Mironescu, {\em Lifting, degree, and
  distributional Jacobian revisited}, Comm. Pure Appl. Math. {\bf 58 }
(2005), 529--551. 
\bibitem{bbn} 
J.~Bourgain, H.~Brezis, and H.-M.~Nguyen, {\em A new estimate for the
  topological degree},  C. R. Math. Acad. Sci. Paris {\bf 340} (2005),  787--791.
\bibitem{bgp}
A.~Boutet~de Monvel-Berthier, V.~Georgescu and R.~Purice, 
 {\em A boundary value problem related to the Ginzburg-Landau model}, 
 Comm. Math. Phys., {\bf 142} (1991), 1--23.
\bibitem{B1} H.~Brezis,  {\em New questions related to the topological degree}, in
  The unity of mathematics, Progr. Math.,  {\bf 244}, Birkh\"auser Boston, Boston, MA, 2006, 137--154. Available at \url{http://www.math.rutgers.edu/~brezis/publications.html}.
\bibitem{bms_coron}
  H.~Brezis, P.~Mironescu and I.~Shafrir, {\em Distances between
    homotopy classes of $W^{s,p}(\sn ;{\mathbb S}^N)$}, ESAIM COCV {\bf 22} (2016), 1204--1235.
\bibitem{bms_w11}
H.~Brezis, P.~Mironescu and I.~Shafrir, {\em Distances between classes
  in $W^{1,1}(\Omega;\so)$}, to appear in Calc. Var. Partial Differential Equations, arXiv:1606.01526.
\bibitem{vmo} H.~Brezis  and L.~Nirenberg, {\em Degree theory and
    BMO. I. Compact manifolds without boundaries}, Selecta
  Math. (N.S.) {\bf 1} (1995), 197-–263. 
\bibitem{LL} E.H.~Lieb and M.~Loss,  Analysis. Second edition. Graduate Studies in Mathematics, 14. American Mathematical Society, Providence, RI, 2001.
\bibitem {pm} P.~Mironescu, {\em Profile decomposition and phase
    control for circle-valued maps in one dimension},
  C. R. Math. Acad. Sci. Paris {\bf 353} (2015), 1087--1092.
\bibitem{ngu} H.-M.~Nguyen, {\em Optimal constant in a new estimate for the degree},
J. Anal. Math. {\bf 101} (2007), 367–395.
\end{thebibliography}
\end{document}